\documentclass[leqno]{article} 

\usepackage[frenchb,english]{babel}
\usepackage[T1]{fontenc}

\usepackage{amsmath}
\usepackage{amssymb}
\usepackage{amsfonts}
\usepackage{enumerate}
\usepackage{vmargin}
\usepackage[all]{xy}
\usepackage{mathrsfs}
\usepackage{mathtools}
\usepackage{lmodern}
\usepackage[colorlinks=true,pagebackref=true,linkcolor=blue]{hyperref}
\usepackage{tocloft}
\usepackage{ulem}
\usepackage[all]{xy}
\usepackage{comment}
\usepackage{centernot}

\setmarginsrb{3cm}{3cm}{3.5cm}{3cm}{0cm}{0cm}{1.5cm}{3cm}

\newcommand{\B}{\mathrm{B}}

\let\L\relax

\newcommand{\L}{\mathrm{L}}
\newcommand{\M}{\mathrm{M}}

\renewcommand{\leq}{\ensuremath{\leqslant}}
\renewcommand{\geq}{\ensuremath{\geqslant}}
\newcommand{\qed}{\hfill \vrule height6pt  width6pt depth0pt}

\newcommand{\norm}[1]{\left\Vert#1\right\Vert}

\newcommand{\co}{\colon}
\newcommand{\ov}{\overset}

\newcommand{\Id}{\mathrm{Id}} 

\newcommand{\op}{\mathrm{op}} 
\let\ker\relax 
\DeclareMathOperator{\ker}{Ker} 
\DeclareMathOperator{\Ran}{Ran} 

\newcommand{\cb}{\mathrm{cb}} 

\newtheorem{thm}{Theorem}[section]

\newtheorem{prop}[thm]{Proposition}
\newtheorem{conj}[thm]{Conjecture}

\newtheorem{cor}[thm]{Corollary}
\newtheorem{lemma}[thm]{Lemma}

\newtheorem{remark}[thm]{Remark}

\newenvironment{proof}[1][]{\noindent {\it Proof #1} : }{\hbox{~}\qed
\smallskip
}

\numberwithin{equation}{section}
\usepackage[nottoc,notlot,notlof]{tocbibind}

\let\OLDthebibliography\thebibliography
\renewcommand\thebibliography[1]{
  \OLDthebibliography{#1}
  \setlength{\parskip}{0pt}
  \setlength{\itemsep}{0pt plus 0.3ex}
}

\begin{document}
\selectlanguage{english}
\title{\bfseries{A characterization of completely bounded normal Jordan $*$-homomorphisms on von Neumann algebras}}
\date{}
\author{\bfseries{C\'edric Arhancet}}

\maketitle


\begin{abstract}
We characterize completely bounded normal Jordan $*$-homomorphisms acting on von Neumann algebras. We also characterize completely positive isometries acting on noncommutative $\L^p$-spaces.     
\end{abstract}




\makeatletter
 \renewcommand{\@makefntext}[1]{#1}
 \makeatother
 \footnotetext{
 The authors are supported by the research program ANR-18-CE40-0021 (project HASCON).  
}

\tableofcontents

\section{Introduction}
\label{sec:Introduction}

The study of linear isometries has been at the heart of the study of Banach spaces since the inception of the field. Indeed, Banach himself stated in his book \cite{Ban1} a result describing the isometries of the $\L^p$-space $\L^p([0,1])$ and the ones of $\ell^p$ where $1 \leq p <\infty$ with $p \not=2$. The proof was clarified by Lamperti in \cite{Lam1} who also extends this result. We refer to the books \cite{FJ1} and \cite{FJ2} for more information on isometries on Banach spaces.

It is natural to examine the isometries acting on von Neumann algebras and more generally noncommutative $\L^p$-spaces associated to von Neumann algebras where $1 \leq p \leq \infty$. Recall that Kadison \cite{Kad1} has proved that a \textit{surjective} isometry $T \co A \to B$ between unital $\mathrm{C}^*$-algebras can be written $T(x)=uJ(x)$ where $J$ is a surjective Jordan $*$-homomorphism and $u \in B$ a unitary. In the paper \cite{Yea1}, Yeadon have succeeded in establishing a description of general (not necessarily surjective) isometries on arbitrary noncommutative $\L^p$-spaces associated with semifinite von Neumann algebras, see \eqref{Yeadon}. This description relies on normal Jordan $*$-homomorphisms. We refer to the papers \cite{ArF3}, \cite{HRW1}, \cite{HSZ1}, \cite{JRS1}, \cite{LMZ1}, \cite{LMZ2}, \cite{She1}, \cite{SuV1}, \cite{Wat2}, \cite{Wat3}, \cite{Wat4} and \cite{Wat5} for more information on isometries on noncommutative $\L^p$-spaces.

It is known from a long time that a positive isometry on a noncommutative $\L^p$-space is not necessarily completely positive or completely bounded. For example, Schatten space $S^p$ the transpose map $S^p \to S^p$, $x \mapsto x^T$ is a positive surjective isometry which is neither completely positive nor completely bounded. It is a striking difference with the world of classical $\L^p$-spaces of measure spaces\footnote{\thefootnote. By \cite[Proposition 2.22]{ArK}, a positive map acting on a classical $\L^p$-space is necessarily completely positive. An isometry on a classical $\L^p$-space with $p \not=2$ is regular in the sense of \cite[(1.1)]{ArK}, hence completely bounded.}. Our first result is Corollary \ref{Cor-cp} which characterizes completely positive isometries on noncommutative $\L^p$-spaces. Our main result Theorem \ref{Th-carac-jordan-homo} characterizes normal $*$-Jordan homomorphisms which are completely bounded.

In Section \ref{sec:preliminaries}, we give background on noncommutative $\L^p$-spaces, Jordan homomorphisms and $n$-minimal operator spaces. In Section \ref{Sec-2_isometries}, we describe completely positive isometries. In Section \ref{Section-cc-isometries}, we prove our main result Theorem \ref{Th-carac-jordan-homo} which characterize completely bounded normal Jordan $*$-homomorphisms in this section. Finally, in Section \ref{Section-cb-isometries}, we state a conjecture on completely bounded isometries acting on noncommutative $\L^p$-spaces associated with finite von Neumann algebras and $p \not=2$.

\section{Preliminaries}
\label{sec:preliminaries}

\paragraph{Noncommutative $\L^p$-spaces} In this paragraph, we provide some background on noncommutative $\L^p$-spaces. Let $M$ be a semifinite von Neumann algebra equipped with a normal semifinite faithful trace $\tau$. We briefly recall the definition of the noncommutative $\L^p$-spaces $\L^p(M)$,  $0< p \leq \infty$, associated with $(\M,\tau)$ and some of their basic properties. The reader is referred to the survey \cite{PiX} and references therein for details and further properties.

If $M$ acts on some Hilbert space $H$, the elements of $\L^p(M)$ can be viewed as closed densely defined (possibly unbounded) operators on $H$. More precisely, let $M^{\prime}$ denote the commutant of $M$ in $\B(H)$. A closed densely defined operator $a$ is said to be affiliated with $M$ if $a$ commutes with every unitary of $M^{\prime}$. An affiliated operator $x$ is called measurable with respect to $(M,\tau)$ if there is a positive number $\lambda > 0$ such that $\tau(1_{[\lambda,\infty)}(|x|))<\infty$, where $1_{[\lambda,\infty)}(|x|)$ is the projection associated with the indicator function of $[\lambda,\infty)$ in the Borel functional calculus of $|x|$. Then the set $\L^0(M)$ of all measurable operators forms a $*$-algebra. We proceed with defining $\L^p(M)$ as a subspace of $\L^0(M)$. First note that for any $x \in \L^0(M)$ and any $0 < p < \infty$, the operator $|x|^p=(x^* x)^{\frac{p}{2}}$ belongs to $\L^0(M)$. If $\L^0(M)^+$ denotes the positive cone of $\L^0(M)$, that is the set of all positive operators in $\L^0(M)$, the trace $\tau$ extends to a positive tracial functional on $\L^0(M)^+$, taking values in $[0,\infty]$, also denoted by $\tau$. For any $0< p < \infty$, the noncommutative $\L^p$-space, $\L^p(M)$, associated with $(M,\tau)$, is
$$
\L^p(M)
\ov{\mathrm{def}}{=} \bigl\{x \in \L^0(M) : \tau(|x|^p) < \infty \bigr\}.
$$
For $x \in \L^p(M)$, let $\norm{x}_p \ov{\mathrm{def}}{=} \tau(|x|^p)^{\frac{1}{p}}$. For $1 \leq p < \infty$, $\norm{\cdot}_p$ defines a complete norm. We let $\L^\infty(M) \ov{\mathrm{def}}{=} M$, equipped with its operator norm $\norm{\,\cdotp}_\infty$.


\paragraph{Jordan homomorphisms} 

Let us first recall some facts on Jordan $*$-homomorphisms that we require in this paper. A Jordan $*$-homomorphism between von Neumann algebras $M$ and $N$ is a linear map $J \co M \to N$ that satisfies $J(x^2)=J(x)^2$ and $J(x^*)=J(x)^*$ for any $x \in M$. 
Any Jordan homomorphism is a positive contraction and any Jordan monomorphism is an isometry.

Let $J \co M \to N$ be a Jordan $*$-homomorphism and let $Z \subset N$ be the von Neumann algebra generated by $J(M)$. Let $e \ov{\mathrm{def}}{=} J(1)$. Then $e$ is a projection and $e$ is the unit of $Z$. According to \cite[Theorem 3.3]{Sto11}, there exist projections $g$ and $f$ in the center of $Z$ such that
\begin{itemize}
	\item [(i)] $g+f=e$.
	\item [(ii)] $x \mapsto J(x)g$ is a $*$-representation.
	\item [(iii)] $x \mapsto J(x)f$ is an anti-$*$-representation. 
\end{itemize}
Let $Z_1 \ov{\mathrm{def}}{=} Zg$ and $Z_2 \ov{\mathrm{def}}{=} Zf$.  We let $\pi \co M \to Z_1$ and $\sigma \co M \to Z_2$ be defined by $\pi(x) \ov{\mathrm{def}}{=} J(x)g$ and $\sigma(x) \ov{\mathrm{def}}{=} J(x)f$, for all $x \in M$. Then, $Z=Z_1 \mathop{\oplus}\limits^{\infty}Z_2$ and $J(x)=\pi(x)+\sigma(x)$, for all $x \in M$. We have the suggestive notations 
\begin{equation}
\label{pisigma}
J
=\begin{bmatrix}
\pi&0\\
0&\sigma
\end{bmatrix}
\qquad \text{and} \qquad
J(a)
=\begin{bmatrix}
\pi(a)&0\\
0&\sigma(a)
\end{bmatrix}
\end{equation}
to refer to such a central decomposition. We note that $J$ is normal (i.e. weak* continuous) if and only if $\pi$ and $\sigma$ are normal. By \cite[page 774]{KaR2}, for any $h,x \in M$, we have
\begin{equation}
\label{Equa-magique-1}
J(hxh)
=hJ(x)h.
\end{equation}

\paragraph{Isometries on noncommutative $\L^p$-spaces} Let $M$ and $N$ be two semifinite von Neumann algebras. For any $1 \leq p < \infty$, consider the associated noncommutative $\L^p$-spaces $\L^p(M)$ and $\L^p(N)$. A remarkable theorem of Yeadon \cite{Yea1} asserts that if $p \not= 2$ and $T \co \L^p(M) \to \L^p(N)$ is a (not necessarily surjective) linear isometry, then there exist an injective normal Jordan $*$-homomorphism $J \co M \to N$, a positive operator $b$ affiliated with $N$ and a partial isometry $w \in N$ such that 
\begin{equation}
\label{Yeadon-complement2}
J(1)
=s(b)
=w^*w.
\end{equation}
(in particular $w^*wb=b$), $J(x)$ commutes with $b$ for all $x \in M$,
and
\begin{equation}
\label{Yeadon}
T(x) 
= wbJ(x),
\end{equation}
for all $x \in M \cap \L^p(M)$ with
\begin{equation}
\label{preservation-trace}
\tau(b^pJ(x))=\tau(x), \quad x \in M_+.
\end{equation}
Moreover, it is easy to check that the proof of \cite[Theorem 2]{Yea1} implies that if $T$ is positive then $w$ is a projection and we have
\begin{equation}
\label{Yeadon-complement}
J(1)
=w.
\end{equation} 

\begin{remark} \normalfont
We warn the reader that the paper \cite{Yea1} contains several errors. The passages ``onto a weakly closed $*$-subalgebra of $\mathcal{A}_2$'' of \cite[Theorem 2]{Yea1} and ``$J(\mathcal{A}$) is a $W^*$-algebra'' of \cite[page 45]{Yea1} are\footnote{\thefootnote. Consider the map $\M_n \mapsto \M_{2n}$, $x \mapsto \begin{bmatrix}
   x  & 0  \\
   0  & x^T  \\
\end{bmatrix}$. The range of a normal Jordan $*$-homomorphism is \textit{not} necessarily  a von Neumann algebra.} false. By pure recopying, theses errors are still unfortunately present in some recent papers.
\end{remark}


\newpage

\paragraph{Local liftings} Recall the following result from \cite{Arh1} and \cite{ArR1}. The last sentence is a new observation which is not difficult and left to the reader.

\begin{thm}
\label{Th-relevement-cp} 
Let $M$ and $N$ be von Neuman algebras. Suppose $1 \leq p < \infty$. Let $T \co \L^p(M) \to \L^p(N)$ be a positive linear map. Let $h$ be a positive element of $\L^p(M)$. Then there exists a unique linear map $v \co M \to s(T(h))Ns(T(h))$ such that 
\begin{equation}
\label{equa-relevement}
T\big(h^{\frac{1}{2}}xh^{\frac{1}{2}}\big)
=T(h)^{\frac{1}{2}}v(x)T(h)^{\frac{1}{2}},\qquad x \in M.
\end{equation}
Moreover, this map $v$ is unital, contractive, and normal. If $T$ is $n$-positive ($1 \leq n \leq \infty)$, then $v$ is also $n$-positive. 
\end{thm}

We say that $v$ is the local lifting of $T$ with respect to $h$.

\vspace{0.2cm}

\paragraph{$n$-minimal operator spaces}

Let $n \geq 1$ be an integer. An operator space $E$ is called $n$-minimal if there exists a (Hausdorff) compact topological space $K$ and a completely isometric map $i \co E \to \mathrm{C}(K,\M_n)$. 

Let $A$ be a $\mathrm{C}^*$-algebra. We denote by $A^\op$ the opposite $\mathrm{C}^*$-algebra. By \cite[Theorem 2.2]{Roy1}, we have the following equivalence  
\begin{equation}
\label{Equivalence-Roydor}
A \text{ is $n$-minimal}
\iff \norm{\Id_A}_{\cb, A \to A^\op} \leq n.
\end{equation}
Recall that a $\mathrm{C}^*$-algebra $A$ is $n$-subhomogeneous \cite[Definition 2.7.6]{BrO1} if all its irreducible representations have dimension at most $n$ and subhomogeneous if $A$ is $n$-subhomogeneous for some $n$. By \cite[Theorem 2.2]{Roy1}, a $\mathrm{C}^*$-algebra $A$ is $n$-minimal if and only if $A$ is $n$-subhomogeneous.

It is known, e.g. \cite[Lemma 2.4]{ShU1}, that a von Neumann algebra $M$ is $n$-subhomogeneous if and only if there exist some distinct integers $n_1,\ldots,n_k \leq n$,  some abelian von Neumann algebras $A_i \not=\{0\}$ and a $*$-isomorphism 
\begin{equation}
\label{AVN-subh}
M
=\M_{n_1}(A_1) \oplus \cdots \oplus \M_{n_1}(A_1).
\end{equation}

\section{Completely positive isometries on noncommutative $\L^p$-spaces}
\label{Sec-2_isometries}


The following result will be used in the proof of our main result.

\begin{prop}
\label{Prop-Local-lifting}
Suppose that $M$ is finite von Neumann algebra and that $N$ is a semifinite von Neumann algebra. Suppose $1 \leq p < \infty$, $p \not=2$. Let $T=wbJ \co \L^p(M) \to \L^p(N)$ be a positive isometry. Then $M \to wNw$, $x \mapsto wJ(x)w$ is the local lifting associated with $T$ and $h=1$ by Theorem \ref{Th-relevement-cp}.
\end{prop}

\begin{proof} 
First, we suppose that $M$ is finite. Note that $1$ belongs to $\L^p(M)$. Since $b$ commute with $J(x)$ for any $x \in M$, we can write
$$
T(x)
=wbJ(x)
=bJ(x)
=b^{\frac{1}{2}}J(x)b^{\frac{1}{2}}, \quad x \in M.
$$ 
Note that $w \ov{\eqref{Yeadon-complement}\eqref{Yeadon-complement2}}{=} s(b)$ is a projection.  
For any $x \in M$, note that
$$
J(x)
=J(1x1)
\ov{\eqref{Equa-magique-1}}{=} J(1)J(x)J(1)
\ov{\eqref{Yeadon-complement}}{=} wJ(x)w.
$$
So we have a well-defined normal map $J \co M \to wNw$ which is unital. Furthermore, we have
\begin{equation}
\label{Magic3}
T(1)
=wbJ(1)
\ov{\eqref{Yeadon-complement}}{=} wbw
=b.
\end{equation}
So for any $x \in M$ we have
$$
T(x)
=b^{\frac{1}{2}}J(x)b^{\frac{1}{2}}
\ov{\eqref{Magic3}}{=} T(1)^{\frac{1}{2}}J(x)T(1)^{\frac{1}{2}}.
$$
Hence $J$ is the local lifting of $T \co \L^p(M) \to \L^p(wNw)$ in the sense of Theorem \ref{Th-relevement-cp} with $h=1$. 
\end{proof}

\begin{remark} \normalfont
Using the local lifting of Theorem \ref{Th-relevement-cp}, it is not difficult to give an alternative proof of Yeadon's theorem for \textit{positive} isometries in the case where $M$ is finite using the well-known result \cite[Lemma 3.3]{Wol1}.
\end{remark}

The following result was announced in \cite[Theorem 2.10]{HRW1}. 

\begin{thm}
\label{thm-iso-cp}
Suppose that $M$ and $N$ are semifinite von Neumann algebras. Suppose $1\leq p < \infty$, $p \not=2$. Let $T=wbJ \co \L^p(M) \to \L^p(N)$ be a 2-positive isometry. Then the associated normal Jordan $*$-homomorphism $J$ is multiplicative.
\end{thm}

\begin{proof} 
First, we suppose that $M$ is finite. By proposition \ref{Prop-Local-lifting}, the map $M \to wNw$ $x \mapsto wJ(x)w$ is the local lifting of $T \co \L^p(M) \to \L^p(N)$ in the sense of Theorem \ref{Th-relevement-cp} with $h=1$. Since $T$ is 2-positive, we infer by Theorem \ref{Th-relevement-cp} that the normal Jordan $*$-monomorphism $J$ is 2-positive. Recall that, by \cite[Corollary 3.2]{Cho1} (see also \cite[Corollary 6]{Gar1} in the non-unital case), a 2-positive unital Jordan $*$-homomorphism between $\mathrm{C}^*$-algebras is a $*$-homomorphism. We conclude that $J$ is a $*$-homomorphism.

If $M$ is semifinite, we consider the net of all $\tau$-finite projections in $M$ equipped with the usual partial order. By the contruction of Yeadon \cite{Yea1}, we know that for any $x \in M$, we have $J(x)=\lim_e J(exe)$ for the strong operator topology. The first part of the proof essentially show that the map $e \mapsto J(exe)$ is 2-positive. An application of \cite[Lemma 2.8]{ArK} (and its proof) shows that $J$ is 2-positive, hence a $*$-homomorphism by \cite[Corollary 3.2]{Cho1}.
\end{proof}

Combining essentially with \cite[Proposition 3.2]{JRS1}, we deduce the following result.

\begin{cor}
\label{Cor-cp}
Suppose that $M$ are $N$ are semifinite von Neumann algebras. Suppose $1 \leq p < \infty$ with $p \not=2$. Let $T=wbJ \co \L^p(M) \to \L^p(N)$ be an isometry. Then the following conditions are equivalent.
\begin{enumerate}
	\item $T$ is 2-positive. 
	\item $T$ is completely positive.
	\item the Jordan $*$-homomorphism $J$ is a $*$-homomorphism.
	\item $T$ is completely isometric.
\end{enumerate}
\end{cor}

\begin{remark} \normalfont
Using \cite{LMZ1} and \cite{HRW1}, we have a similar result for completely positive disjointness preserving operators acting on noncommutative $\L^p$-spaces. Details are left to the reader.
\end{remark}

\begin{remark} \normalfont
With \cite{JRS1}, we can give a variant of the result for the case where $N$ is $\sigma$-finite.
\end{remark}

\section{Completely bounded normal $*$-Jordan homomorphisms}
\label{Section-cc-isometries}

We will use the following crucial observation.

\begin{lemma}
\label{Petit-lemme-wonderful}
Suppose that $A$ and $B$ are $\mathrm{C}^*$-algebras. Suppose that $n \geq 1$ is an integer. Let $\sigma \co A \to B$ be an anti-$*$-homomorphism which is completely bounded with $\norm{\sigma}_{\cb, A \to B} \leq n$. Then the $\mathrm{C}^*$-algebra $A/\ker \sigma$ is $n$-minimal.
\end{lemma}

\begin{proof}
We let $I \ov{\mathrm{def}}{=} \ker \sigma$. By \cite[Corollary 1.8.3]{Dix}, we have two faithful anti-$*$-homomorphisms $\tilde{\sigma} \co A/I \to \Ran \sigma$ and $\tilde{\sigma}^{-1} \co \Ran \sigma \to A/I$. By \cite[1.2.15]{BLM1}, we have
\begin{equation}
\label{First-estimate-cb}
\norm{\tilde{\sigma}}_{\cb,A/I \to \Ran \sigma}
=\norm{\sigma}_{\cb,A \to B} \leq n.
\end{equation}
Note that the map $\tilde{\sigma}^{-1} \co \Ran \sigma \to (A/I)^\op$ is a $*$-homomorphism. By \cite[Proposition 1.2.4]{BLM1} we obtain the estimate
\begin{equation}
\label{Second-estimate-cb}
\norm{\tilde{\sigma}^{-1}}_{\cb,\Ran \sigma \to (A/I)^\op}
\leq 1.
\end{equation}
Hence, we deduce that
\begin{align*}
\MoveEqLeft
\norm{\Id_{A/I}}_{\cb,A/I \to (A/I)^\op}
= \norm{\tilde{\sigma}^{-1}\tilde{\sigma}}_{\cb,A/I \to (A/I)^\op} \\
&\leq \norm{\tilde{\sigma}^{-1}}_{\cb,\Ran \sigma \to (A/I)^\op} \norm{\tilde{\sigma}}_{\cb,A/I \to \Ran \sigma}
\ov{\eqref{First-estimate-cb} \eqref{Second-estimate-cb}}{\leq} n.
\end{align*} 
By the characterization \eqref{Equivalence-Roydor}, we conclude that the quotient $A/I$ is $n$-minimal.
\end{proof}

Now, we characterize completely bounded normal Jordan $*$-homomorphisms.

\begin{thm}
\label{Th-carac-jordan-homo}
Suppose that $M$ and $N$ are von Neumann algebras. Suppose that $n \geq 1$ is an integer. Let $J \co M \to N$ be a normal $*$-Jordan-homomorphism. Then $J$ is completely bounded if and only if there exist a decomposition $M=M_1 \oplus M_2$ where $M_1$ and $M_2$ are von Neumann algebras and an integer $n \geq 1$ such that the restriction $J|M_1 \co M_1 \to N$ is a normal $*$-homomorphism and such that the von Neumann algebra $M_2$ is $n$-minimal.
\end{thm}

\begin{proof}
$\Rightarrow$: Suppose that $J$ is completely bounded. Consider some integer $n \geq 1$ such that $\norm{J}_{\cb,M \to N} \leq n$. Let $J=\pi +\sigma$ be the canonical decomposition of the normal Jordan $*$-momomorphism $J$ given by \eqref{pisigma}. By composition, we infer that the anti-$*$-homomorphism $\sigma \co M \to N$, $a \mapsto J(a)f$ is also completely bounded with $\norm{\sigma}_{\cb,M \to N} \leq n$. By Lemma \ref{Petit-lemme-wonderful}, we conclude that the quotient $M/\ker \sigma$ is $n$-minimal.

Note that $\ker \sigma$ is a weak* closed two-sided ideal of $M$. By \cite[Theorem 6.8.8]{KaR2}, there exists a central projection $q$ of $M$ such that $\ker \sigma = qM$. We let $M_1 \ov{\mathrm{def}}{=} (1-q)M$ and $M_2 \ov{\mathrm{def}}{=} \ker \sigma $. We have the decomposition 
\begin{equation}
\label{Inter-108}
M
=(1-q)M \oplus qM
=M_1 \oplus M_2.
\end{equation} 
It is plain that that $J|M_1 \co M_1 \to N$ is a $*$-homomorphism. Now, it suffices to note that we have a $*$-isomorphism $M_1 \ov{\eqref{Inter-108}}{=} M/M_2=M/\ker \sigma$.

$\Leftarrow$: Note that using the decomposition of the Jordan $*$-monomorphism $J| M_2 \co M_2 \to N$, we can write $J|M_2=\pi+\sigma$ where $\pi \co M_2 \to N$ is a normal $*$-homomorphism and where $\sigma \co M_2^\op \to N$ is a normal $*$-homomorphism. Since $M_2$ is $n$-minimal, by \eqref{Equivalence-Roydor}, we know that the identity map $\Id_{M_2} \co M_2 \to M_2^\op$ is completely bounded. By composition, we infer that the anti-$*$-homomorphism $\sigma \co M_2 \to N$ is completely bounded. We deduce that the restriction $J| M_2 \co M_2 \to N$ of $J$ on $M_2$ is completely bounded since $\pi$ is completely contractive by \cite[Proposition 1.2.4]{BLM1}. Note that again by \cite[Proposition 1.2.4]{BLM1} the normal $*$-homomorphism $J|M_1 \co M_1 \to N$ is completely contractive. We conclude that $J$ is completely bounded.
\end{proof}

\section{A conjecture on completely bounded isometries}
\label{Section-cb-isometries}

We state the folowing conjecture for the structure of completely bounded isometries acting on noncommutative $\L^p$-spaces associated with \textit{finite} von Neumann algebras and $p \not=2$.

\begin{conj}
\label{Th-main}
Suppose that $M$ is a finite von Neumann algebra and that $N$ is a semifinite von Neumann algebra. Suppose $1 \leq p< \infty$, $p \not=2$. Let $T \co \L^p(M) \to \L^p(N)$ be a (not necessarily surjective) isometry. Then $T$ is completely bounded if and only if there exists a decomposition $M=M_1 \oplus M_2$ where $M_1$ and $M_2$ are von Neumann algebras and an integer $n \geq 1$ such that the restriction $T|\L^p(M_1) \co \L^p(M_1) \to \L^p(N)$ is a complete isometry and such that the von Neumann algebra $M_2$ is $n$-minimal.
\end{conj}

The implication $\Leftarrow$ is true by the following reasoning. Note that using the decomposition of the normal Jordan $*$-monomorphism of the isometry $T| \L^p(M_2) \co \L^p(M_2) \to \L^p(N)$, we can write $T| \L^p(M_2)=T_1+T_2$ where $T_1 \co \L^p(M_2) \to \L^p(N)$ and $T_2 \co \L^p(M_2)^\op \to \L^p(N)$ are completely contractive. By \eqref{AVN-subh}, there exist some distinct integers $n_1,\ldots,n_k \leq n$, some abelian von Neumann algebras $A_i \not=\{0\}$ and a $*$-isomorphism
\begin{equation*}
\label{AVN-subhomogeneus-2}
M
=\M_{n_1}(A_1) \oplus \cdots \oplus \M_{n_1}(A_1).
\end{equation*} 
So we have an isometrical isomorphism
\begin{equation*}
\label{AVN-subhomogeneus}
\L^p(M_2)
=S^p_{n_1}(\L^p(A_1)) \oplus_p \cdots \oplus_p S^p_{n_1}(\L^p(A_1)).
\end{equation*} 
It is folklore (and elementary) that the completely norm of the transpose map $t_n \co S^p_n \to S^p_n$, $x \mapsto x^T$ satisfies $\norm{t_n}_{\cb,S^p_n \to S^p_n} = n^{|1-\frac{2}{p}|}$. Using this observation and \cite[page 2]{Jun}, it is obvious that the identity map $\Id_{\L^p(M_2)} \co \L^p(M_2) \to \L^p(M_2)^\op$ is completely bounded. We infer that $T_2 \co \L^p(M_2) \to \L^p(N)$ is completely bounded. So the restriction $T| \L^p(M_2) \co \L^p(M_2) \to \L^p(M)$ of $T$ on $\L^p(M_2)$ is completely bounded. Since $T|\L^p(M_1) \co \L^p(M_1) \to \L^p(N)$ is completely contractive, we conclude that $T$ is completely bounded.



For the reverse direction, it is quite easy to see that it suffices to prove the following.

\begin{conj}
\label{Conj-fundamental}
Suppose that $M$ is a finite von Neumann algebra and that $N$ is a semifinite von Neumann algebra. Suppose $1 \leq p< \infty$, $p \not=2$. Let $T \co \L^p(M) \to \L^p(N)$ be an isometry with a Yeadon factorization $T=wbJ$. If $T$ is completely bounded then the normal $*$-Jordan homomorphism $J$ is completely bounded.
\end{conj}

\begin{remark} \normalfont
Note that the isometric Schur multipliers on $S^p$ are described in \cite{Arh3} and it is easy to check that these multipliers are completely contractive. 
\end{remark}

\textbf{Acknowledgment}.
I will thank Christian Le Merdy for a short discussion on this topic.



\vspace{0.2cm}
\footnotesize{
\noindent C\'edric Arhancet\\ 
\noindent13 rue Didier Daurat, 81000 Albi, France\\
URL: \href{http://sites.google.com/site/cedricarhancet}{https://sites.google.com/site/cedricarhancet}\\
cedric.arhancet@protonmail.com\\

}

\end{document}